\documentclass[12pt]{amsart}
\usepackage{amsfonts, amssymb, amsmath, amsthm}
\usepackage[hmargin=1.25in, vmargin=1.25in]{geometry}

\newtheorem{theorem}{Theorem}
\newtheorem{proposition}{Proposition}
\newtheorem{lemma}{Lemma}
\newtheorem{corollary}{Corollary}

\theoremstyle{definition}

\theoremstyle{remark}
\newtheorem{remark}{Remark}

\begin{document}

\title[Rotation invariant ultradistributions] {Rotation invariant ultradistributions}

\begin{abstract} We prove that an ultradistribution is rotation invariant if and only if it coincides with its spherical mean. For it, we study the problem of spherical representations of ultradistributions on $\mathbb{R}^{n}$. Our results apply to both the quasianalytic and the non-quasianalytic case.
\end{abstract}

\author[\DJ. Vu\v{c}kovi\'{c}]{\DJ or\dj e Vu\v{c}kovi\'{c}}
\address{Department of Mathematics, Ghent University, Krijgslaan 281 Gebouw S22, 9000 Gent, Belgium}
\email{dordev@cage.UGent.be}

\author[J. Vindas]{Jasson Vindas}
\thanks{The authors gratefully acknowledge support by Ghent University, through the BOF-grant 01N01014.}

\address{Department of Mathematics, Ghent University, Krijgslaan 281 Gebouw S22, 9000 Gent, Belgium}
\email{jasson.vindas@UGent.be}

\maketitle

\section{Introduction}
Rotation invariant generalized functions have been studied by several authors, see e.g. \cite{Chung,T1960,Vernaeve2008}. The 
problem of the characterization  of rotation invariant ultradistributions and hyperfunctions was considered by Chung and Na in  \cite{Chung}. They showed there that a non-quasianalytic  ultradistribution or a hyperfunction is rotation invariant if and only if it is equal to its spherical mean. For continuous functions this result is clear, as a rotation invariant function must be radial and its spherical mean is given by
$$\varphi_S(x)=\frac{1}{|\mathbb{S}^{n-1}|}\int_{\mathbb S^{n-1}} \varphi(|x|\omega)d\omega.$$

The approach of Chung and Na to the problem consists in reducing the case of rotation invariant generalized functions to that of ordinary functions. For ultradistributions, non-quasianalyticity was a crucial assumption for their method since they regularized by convolving with a net of compactly supported ultradifferentiable mollifiers. In the hyperfunction case they applied a similar idea but this time based on Matsuzawa's heat kernel method.    

The aim of this article is to show that the characterization of rotation invariant ultradistributions in terms of their spherical means remains valid for quasianalytic ultradistributions. Our approach differs from that of Chang and Na, and we also recover their results for non-quasianalytic ultradistributions and hyperfunctions. 

Our method is based upon the study of spherical representations of ultradistributions, that is, the problem of representing an ultradistribution $f$ on $\mathbb{R}^{n}$ by an ultradistribution $g$ on $\mathbb{R}\times \mathbb{S}^{n-1}$ in such a way that $\langle f(x),\varphi(x)=\langle g(r,\omega),\varphi(r\omega)\rangle$. Spherical representations of distributions were studied by Drozhzhinov and Zav'yalov in \cite{DZ2006}. We shall also exploit results on spherical harmonic expansions of ultradifferentiable functions and ultradistributions on the unit sphere $\mathbb{S}^{n-1}$, recently obtained by us in \cite{VV2}. We mention that the theory of spherical harmonic expansions of distributions was developed by Estrada and Kanwal in \cite{Estrada}. 

The plan of the article is as follows. Section \ref{preli} discusses some background material on spherical harmonics and ultradistributions. Spherical representations of ultradistributions are studied in Section \ref{section spherical representation}. We show in Section \ref{section rotation invariant} that any ultradistribution is rotation invariant if and only if it coincides with its spherical mean. In the quasianalytic case we go beyond quasianalytic functionals by employing sheaves of quasianalytic ultradistributions.

\section{Preliminaries and Auxiliary Results}\label{preli}
In this section we collect some useful concepts and auxiliary results that will play a role in our study of rotation invariant ultradistributions and spherical representations.
\subsection{Spherical Harmonics} The theory of spherical harmonics is a classical subject in analysis and it is very well explained in several textbooks (see e.g. \cite[Chap.~5]{Axler}). A spherical harmonic of degree $j$ is simply the restriction to the Euclidean unit sphere $\mathbb{S}^{n-1}$ of a harmonic homogeneous polynomial of degree $j$ on $\mathbb{R}^{n}$. Let ${\mathcal H}_j(\mathbb{S}^{n-1})$ be the space of spherical harmonics of degree $j$. The dimension $d_{j}$ of ${\mathcal H}_j(\mathbb{S}^{n-1})$ can be explicitly calculated \cite[Prop.~5.8]{Axler}; although we will not make use of the explicit value, we need the growth estimate $d_{j}\asymp j^{n-2}$. We also point out that each $\mathcal{H}_{j}(\mathbb{S}^{n-1})$ is invariant under the action of the orthogonal group $O(n)$.

It is well-known \cite{Axler} that
$$
L^{2}(\mathbb{S}^{n-1})=\bigoplus_{j=0}^{\infty} \mathcal{H}_{j}(\mathbb{S}^{n-1}),
$$
where the $L^2$-inner product is taken with respect to the surface measure of $\mathbb{S}^{n-1}$.

Through the rest of the article we fix an orthonormal basis $\{Y_{k,j}\}_{k=1}^{d_{j}}$ of each  ${\mathcal H}_j(\mathbb{S}^{n-1})$, consisting of real-valued spherical harmonics. Hence, every function $f\in L^2(\mathbb{S}^{n-1})$ can be expanded as 
$$f(\omega)=\sum_{j=0}^{\infty}\sum_{k=1}^{d_{j}}c_{k,j}Y_{k,j}(\omega)$$
with convergence in $L^2(\mathbb{S}^{n-1})$.

\subsection{Ultradistributions} We briefly review in the subsection some properties of the spaces of ultradifferentiable functions and ultradistributions \cite{PilipovicK,Komatsu,Komatsu2}.

We fix a positive sequence $(M_p)_{p \in \mathbb{N}}$ with $M_{0}=1$. We will make use of some of the following standard conditions on the weight sequence
\begin{itemize}
\item [$(M.0)\:$] $p!\subset M_p$ in the Roumieu case, or $p!\prec M_p$ in the Beurling case.
\item [$(M.1)\:$] $M^{2}_{p}\leq M_{p-1}M_{p+1},$  $p\geq 1$.
\item [$(M.2)'$]$M_{p+1}\leq A H^p M_p$, $p\in\mathbb{N}$, for some $A,H>0$.
\item [$(M.2)\:$] $ M_{p}\leq A H^p\min_{1\leq q\leq p} \{M_{q} M_{p-q}\},$ $p\in\mathbb{N}$, for some $A,H>0$.
\item [$(M.3)'$]  $ \sum_{p=1}^{\infty}M_{p-1}/M_p<\infty$. 
\item [$(QA)$\:] $ \sum_{p=1}^{\infty}M_{p-1}/M_p=\infty$. 
\end{itemize}

We refer to \cite{Komatsu} for a detailed explanation of the meaning of all these conditions. The relations $\subset$ and $\prec$ used in $(M.0)$ are defined as follows. One writes $N_p\subset M_p$ ($N_p\prec M_p$) if there are $C,\ell>0$ (for each $\ell$ there is $C=C_{\ell}$) such that 
$N_p\leq C\ell^{p}M_{p},$ $p\in\mathbb{N}$. If $(M.3)'$ holds, we call $M_p$ non-quasianalytic; otherwise it is said to be quasianalytic. The associated function of the sequence is defined as
$$M(t)=\sup_{p\in\mathbb N}\log\left(\frac{t^p }{M_p}\right), \quad t> 0.$$
In the particular case of Gevrey sequences $M_p=(p!)^{s}$, the associated function is $M(t)\asymp t^{1/s}$ \cite{GS}.

Let $\Omega \subseteq \mathbb{R}^d$ be open. The space of all $C^{\infty}$-functions on $\Omega$ is denoted by $\mathcal{E}(\Omega)$. For $K \Subset \Omega$ (a compact subset with non-empty interior) and $h > 0$, one writes $\mathcal{E}^{\{M_p\},h}(K)$ for the space of all $\varphi \in \mathcal{E}(\Omega)$ such that
\[ \| \varphi \|_{\mathcal{E}^{\{M_p\},h}(K)} := \sup_{\substack{x \in K \\ \alpha \in \mathbb{N}^n}} \frac{|\varphi^{(\alpha)}(x)|}{h^{|\alpha|}M_{|\alpha|}} < \infty, \]
and $\mathcal{D}^{\{M_p\},h}_{K}$ stands for the closed subspace of $\mathcal{E}^{\{M_p\},h}(K)$ consisting of functions with compact support in $K$ (by the Denjoy-Carleman theorem, its non-triviality is equivalent to $(M.3)'$). Set then

\[  \mathcal{E}^{\{M_p\}}(\Omega) = \varprojlim_{K \Subset \Omega} \varinjlim_{h \rightarrow \infty} \mathcal{E}^{\{M_p\},h}(K), \qquad \mathcal{E}^{(M_p)}(\Omega) = \varprojlim_{K \Subset \Omega} \varprojlim_{h \rightarrow 0^+} \mathcal{E}^{\{M_p\},h}(K),   \]
and
\[ \mathcal{D}^{\{M_p\}}(\Omega) = \varinjlim_{K \Subset \Omega} \varinjlim_{h \rightarrow \infty} \mathcal{D}^{\{M_p\},h}_{K} \quad \mathcal{D}^{(M_p)}(\Omega) = \varinjlim_{K \Subset \Omega} \varprojlim_{h \rightarrow 0^+} \mathcal{D}^{\{M_p\},h}_{K}. \]
Their duals are the spaces of ultradistributions of Roumieu and Beurling type \cite{Komatsu}. 

In order to treat these spaces simultaneously we write $\ast=\{M_p\},(M_p)$. In statements needing a separate treatment we will first state assertions for the Roumieu case, followed by the Beurling one in parenthesis. 

In the important case $\ast=\{p!\}$, we write $\mathcal{A}(\Omega)={\mathcal{E}}^{\{p!\}}(\Omega)$, the space of real analytic functions on $\Omega$; its dual $\mathcal{A}'(\Omega)$ is then the space of analytic functionals on $\Omega$. Note that $(M.0)$ implies that $\mathcal{A}(\Omega)\subseteq \mathcal{E}^{\ast}(\Omega)$, and, if in addition $(M.1)$ and $(M.2)'$ hold, $\mathcal{A}(\Omega)$ is densely injected into $\mathcal{E}^{\ast}(\Omega)$ because the polynomials are dense in both spaces; in particular, ${\mathcal{E}^{\ast}}'(\Omega)\subseteq \mathcal{A}'(\Omega)$ under these assumptions.

If one assumes $(M.0)$ (as we will always do), the pullback of an invertible analytic change of variables $\Omega\to U$ becomes a TVS isomorphism between $\mathcal{E}^{\ast}(U)$ and $\mathcal{E}^{\ast}(\Omega)$ \cite[Prop.~8.4.1]{Hormander}. Therefore, one can always define the spaces $\mathcal{E}^{\ast}(M)$ and ${\mathcal{E}^{\ast}}'(M)$ for $\sigma$-locally compact analytic manifolds $M$ via charts if $(M.0)$ holds. Note that $(M.0)$ is automatically fulfilled if $(M.1)$ and $(M.3)'$ hold \cite[Lemma~4.1]{Komatsu}.

\subsection{Ultradistributions on $\mathbb{S}^{n-1}$ and Spherical Harmonics}
 \label{section ultradistributions and spherical harmonics}
The spaces of ultradifferentiable functions and ultradistributions on $\mathbb{S}^{n-1}$ can be described in terms of spherical harmonic expansions. A proof of the following theorem will appear in our forthcoming paper \cite{VV2}, which also deals with ultradistributional boundary values of harmonic functions on the sphere. We point out that the distribution case goes back to Estrada and Kanwal \cite{Estrada}. See also the forthcoming article \cite{dasgupta-ruzhansky2016} for a treatment of the problem on compact analytic manifolds.  We will apply Theorem \ref{t1 sh} in the next subsection to expand ultradifferentiable functions and ultradistributions on $\mathbb{R}\times \mathbb{S}^{n-1}$ in spherical harmonic series.

\begin{theorem}[\cite{VV2}] 
\label{t1 sh}
Suppose that $M_p$ satisfies $(M.0)$, $(M.1)$, and $(M.2)'$.
\begin{itemize}
\item [$(i)$] Let  $\varphi\in L^{2}({\mathbb S}^{n-1})$ have spherical harmonic expansion
\begin{equation}
\label{sheq1}
\varphi(\omega)=\sum_{j=0}^{\infty}\sum_{k=1}^{d_j}a_{k,j}Y_{k,j}(\omega).
\end{equation}
Then, $\varphi\in\mathcal{E}^{\ast}(\mathbb{S}^{n-1})$ if and only if the estimate
\begin{equation}
\label{sheq2}
\sup_{k,j}|a_{k,j}| e^{M\left(\frac{j}{h}\right)}<\infty
\end{equation}
holds for some $h>0$ (for all $h>0$).
\item [$(ii)$] Every ultradistribution $f\in{\mathcal{E}^{\ast}}'(\mathbb{S}^{n-1})$ admits a spherical harmonic expansion
\begin{equation}
\label{sheq3}
f(\omega)=\sum_{j=0}^{\infty}\sum_{k=1}^{d_j}c_{k,j}Y_{k,j}(\omega),
\end{equation}
where the coefficients satisfy the estimate
\begin{equation}
\label{sheq4}
\sup_{k,j}|c_{k,j}| e^{-M\left(\frac{j}{h}\right)}<\infty
\end{equation}
for each $h>0$ (for some $h>0$). Conversely, any series $(\ref{sheq3})$ converges in ${\mathcal{E}^{\ast}}'(\mathbb{S}^{n-1})$ if the coefficients have the stated growth properties.
\end{itemize}
\end{theorem}

It is important to point out that Theorem \ref{t1 sh} as stated above does not reveal all topological information encoded by the spherical harmonic coefficients. Denote as $\mathcal{E}^{\{M_p\},h}_{sh}(\mathbb{S}^{n-1})$ the Banach space of all (necessarily smooth) functions $\varphi$ on $\mathbb{S}^{n-1}$ having spherical harmonic expansion with coefficients $a_{k,j}$ satisfying (\ref{sheq2}) for a given $h$. One can then show \cite{VV2}
\[ \mathcal{E}^{\{M_p\}}(\mathbb{S}^{n-1}) =  \varinjlim_{h \rightarrow \infty} \mathcal{E}^{\{M_p\},h}_{sh}(\mathbb{S}^{n-1}) \quad \mbox{ and } \quad \mathcal{E}^{(M_p)}(\mathbb{S}^{n-1})= \varprojlim_{h \rightarrow 0^+} \mathcal{E}^{\{M_p\},h}_{sh}(\mathbb{S}^{n-1})  \]
topologically. This for instance yields immediately the nuclearity of $\mathcal{E}^{\ast}(\mathbb{S}^{n-1})$ under the assumptions of Theorem \ref{t1 sh}. Observe also that the norm on the Banach space $\mathcal{E}^{\{M_p\},h}_{sh}(\mathbb{S}^{n-1})$ can be rewritten as
\begin{equation}\label{sheq5}
\|\varphi \|_{\mathcal{E}^{\{M_p\},h}_{sh}(\mathbb{S}^{n-1})}= \sup_{k,j} e^{M\left(\frac{j}{h}\right)}\left|\int_{\mathbb{S}^{n-1}}\varphi(\omega)Y_{k,j}(\omega)d\omega\right|.
\end{equation}
A similar topological description can be given for the ultradistribution space ${\mathcal{E}^{\ast}}'(\mathbb{S}^{n-1})$ by using the coefficient estimates (\ref{sheq4}).

\subsection{Ultradistributions on $\mathbb{R}\times\mathbb{S}^{n-1}$} We also need some properties of the spaces $\mathcal{E}^{\ast}(\mathbb{R}\times\mathbb{S}^{n})$ and ${\mathcal{E}^{\ast}}'(\mathbb{R}\times\mathbb{S}^{n})$. Let us assume $(M.0)$, $(M.1)$, and $(M.2)$. We have
$$
\mathcal{E}^{\ast}(\mathbb{R}\times\mathbb{S}^{n-1})=\mathcal{E}^{\ast}(\mathbb{R}, \mathcal{E}^{\ast}(\mathbb{S}^{n-1}))=\mathcal{E}^{\ast}(\mathbb{S}^{n-1}, \mathcal{E}^{\ast}(\mathbb{R}))=\mathcal{E}^{\ast}(\mathbb{R})\widehat{\otimes}\mathcal{E}^{\ast}(\mathbb{S}^{n-1}),
$$
where the tensor product may be equally taken with respect to the $\pi$- or $\epsilon$-topology in view of the nuclearity of these spaces. In fact, the first two equalities are completely trivial, while the third one follows because the linear span of terms of the form $p\otimes Y$, where $p$ is a polynomial on $\mathbb{R}$ and $Y$ a spherical harmonic, is dense in  $\mathcal{E}^{\ast}(\mathbb{R}\times\mathbb{S}^{n-1})$. Moreover, this immediately gives (cf. (\ref{sheq4})) that 
\[  \mathcal{E}^{\{M_p\}}(\mathbb{R}\times\mathbb{S}^{n-1}) = \varprojlim_{K \Subset \mathbb{R}} \varinjlim_{h \rightarrow \infty} \mathcal{E}^{\{M_p\},h}_{sh}(K\times\mathbb{S}^{n-1}) 
\]
and
\[
\mathcal{E}^{(M_p)}(\mathbb{R}\times\mathbb{S}^{n-1}) = \varprojlim_{K \Subset \mathbb{R}} \varprojlim_{h \rightarrow 0^+} \mathcal{E}^{\{M_p\},h}_{sh}(K\times\mathbb{S}^{n-1}),   \]
where  $\mathcal{E}^{\{M_p\},h}_{sh}(K\times\mathbb{S}^{n-1})$ is the space of functions $\Phi$ such that
\begin{equation}
\label{sheq5}
\|\Phi \|_{\mathcal{E}^{\{M_p\},h}_{sh}(K\times\mathbb{S}^{n-1})}= \sup_{k,j} e^{M\left(\frac{j}{h}\right)}\left\|\int_{\mathbb{S}^{n-1}}\Phi(\:\cdot \:,\omega)Y_{k,j}(\omega)d\omega\right\|_{\mathcal{E}^{\{M_p\},h}(K)}<\infty.
\end{equation}

These comments yield the following proposition.
\begin{proposition}\label{p sh} Assume $M_p$ satisfies $(M.0)$, $(M.1)$, and $(M.2).$ 
\begin{itemize} 
\item [$(i)$] Every $\Phi\in \mathcal{E}^{\ast}(\mathbb{R}\times\mathbb{S}^{n-1})$ has convergent expansion
$$
\Phi(r,\omega)=\sum_{j=0}^{\infty}\sum_{k=1}^{d_j}a_{k,j}(r)Y_{k,j}(\omega) \quad \mbox{in }\mathcal{E}^{\ast}(\mathbb{R}\times\mathbb{S}^{n-1}),
$$
where $a_{k,j}\in\mathcal{E}^{\ast}(\mathbb{R})$ and for each $K\Subset\mathbb{R}$
\begin{equation}
\label{sheq6}
\sup_{k,j}e^{M\left(\frac{j}{h}\right)}\|a_{k,j}\|_{\mathcal{E}^{\{M_p\},h}(K)} <\infty
\end{equation}
for some $h>0$ $($for all $h>0$$)$. Conversely, any such series converges in the space $\mathcal{E}^{\ast}(\mathbb{R}\times\mathbb{S}^{n-1})$ if $(\ref{sheq6})$ holds.

\item [$(ii)$] Every ultradistribution $g\in{\mathcal{E}^{\ast}}'(\mathbb{R}\times\mathbb{S}^{n-1})$ has convergent expansion 
$$
g(r,\omega)=\sum_{j=0}^{\infty}\sum_{k=1}^{d_j}c_{k,j}(r)\otimes Y_{k,j}(\omega) \quad \mbox{in }{\mathcal{E}^{\ast}}'(\mathbb{R}\times\mathbb{S}^{n-1}),
$$
where $c_{k,j}\in{\mathcal{E}^{\ast}}'(\mathbb{R})$ and for any bounded subset $B\subset \mathcal{E}^{\ast}(\mathbb{R})$ one has
\begin{equation}
\label{sheq7}
\sup_{k,j}\: e^{-M\left(\frac{j}{h}\right)}\sup_{\varphi\in B} |\langle c_{k,j},\varphi \rangle |<\infty.
\end{equation}
for each $h$ (for some $h$).
Conversely, any such series converges in the space ${\mathcal{E}^{\ast}}'(\mathbb{R}\times\mathbb{S}^{n-1})$ if $(\ref{sheq7})$ holds.
\end{itemize}
\end{proposition}
\begin{proof} For $(i)$, simply note that $a_{k,j}(r)=\int_{\mathbb{S}^{n-1}}\Phi(r,\omega)Y_{k,j}(\omega)d\omega$ and so (\ref{sheq5}) is the same as (\ref{sheq6}). The convergence of the series expansions of $\Phi$ is trivial to check via the seminorms (\ref{sheq5}). Part $(ii)$ follows from $(i)$ and the canonical identification ${\mathcal{E}^{\ast}}'(\mathbb{R}\times\mathbb{S}^{n-1})={\mathcal{E}^{\ast}}'(\mathbb{S}^{n-1}, {\mathcal{E}^{\ast}}'(\mathbb{R}))(:=L_{b}(\mathcal{E}^{\ast}(\mathbb{S}^{n-1}),{\mathcal{E}^{\ast}}'(\mathbb{R})))$.
\end{proof}

Note that the same proposition holds for ${\mathcal{D}^{\ast}}'(\mathbb{R}\times\mathbb{S}^{n})$ if one additionally assumes $(M.3)'$.

\section{Spherical Representations of Ultradistributions}
\label{section spherical representation}

It is easy to see that any $g\in {\mathcal{E}^{\ast}}'(\mathbb{R}\times \mathbb{S}^{n-1})$ gives rise to an ultradistribution $f$ on $\mathbb{R}^{n}$ via the formula
\begin{equation}
\label{sreq1} 
\langle f(x),\varphi(x) \rangle = \langle g(r,\omega),\varphi(r\omega) \rangle.
\end{equation}
In fact, the assignment $g\mapsto f$ is simply the transpose of 
\begin{equation}
\label{sreq2} 
\varphi \mapsto \Phi, \quad \Phi(r,\omega):=\varphi(r\omega),
\end{equation}
which is obviously continuous $\mathcal E^{*}(\mathbb R^n)\to \mathcal{E}^{\ast}(\mathbb{R}\times \mathbb{S}^{n-1})$.

In this section we study the converse representation problem. That is, the problem of representing an $f\in{\mathcal E^{*}}'(\mathbb R^n)$ as in (\ref{sreq1}) for some ultradistribution $g$ on $\mathbb{R}\times\mathbb{S}^{n-1}$. We shall call any such $g$ a \emph{spherical representation} of $f$. Naturally, the same considerations make sense for $f\in{\mathcal D^{*}}'(\mathbb R^n)$ in the non-quasianalytic case. 

In order to fix ideas, let us first discuss the distribution case. The problem of finding a spherical representation of $f\in\mathcal{D}'(\mathbb{R}^{n})$ can be reduced to the determination of the image of $\mathcal{E}(\mathbb{R}^{n})$ under the mapping (\ref{sreq2}). Notice that the range of this mapping is obviously contained in the subspace of ``even'' test functions, namely,
$${\mathcal E}_e(\mathbb R\times \mathbb S^{n-1})=\{\Phi\in {\mathcal E}(\mathbb R\times\mathbb S^{n-1}):\: \Phi(-r,-\omega)=\Phi(r,\omega), \forall (r,\omega)\in\mathbb{R}\times\mathbb{S}^{n-1}\}.$$
In other words, one is interested here in characterizing all those $\Phi\in \mathcal{E}_{e}(\mathbb{R}\times \mathbb{S}^{n-1})$ such that 
\begin{equation}
\label{sreq3} 
\varphi(x)=\Phi\left(|x|, \frac{x}{|x|}\right)
\end{equation}
is a smooth function on $\mathbb{R}^{n}$. The solution to the latter problem is well-known:
\begin{proposition}[\cite{DZ2006,HelgasonBook}]
\label{srp1} Let $\Phi\in \mathcal{E}_{e}(\mathbb{R}\times \mathbb{S}^{n-1})$. Then, $\varphi$ given by $(\ref{sreq3})$ is an element of $\mathcal{E}(\mathbb{R}^{n})$  if and only if  $\Phi$ has the property that for each $m\in\mathbb{N}$
\begin{equation}
\label{sreq4}
\frac{\partial^m\Phi}{\partial r^m}(0,\omega) \ \ \mbox{is a homogeneous polynomial of degree } 
 m.
\end{equation}
\end{proposition}

Write
$$
\mathcal{V}(\mathbb{R}\times \mathbb{S}^{n-1}):=\{\Phi\in {\mathcal E}_e(\mathbb R\times \mathbb S^{n-1}): \: (\ref{sreq4}) \ \mbox{holds for each }m\in\mathbb{N}\}.
$$
Hence $\mathcal{V}(\mathbb{R}\times \mathbb{S}^{n-1})$ is precisely the image of $\mathcal{E}(\mathbb{R}^{n})$ under (\ref{sreq2}). Since it is obviously a closed subspace of ${\mathcal E}(\mathbb R\times \mathbb S^{n-1})$, one obtains from the open mapping theorem that $\mathcal{E}(\mathbb{R}^{n})$ is isomorphic to  $\mathcal{V}(\mathbb{R}\times \mathbb{S}^{n-1})$ via 
(\ref{sreq2}). Given $f\in\mathcal{D}'(\mathbb{R}^{n})$, $\langle f(x),\Phi(|x|, x/|x|)\rangle$ defines a continuous linear functional on $\mathcal{D}(\mathbb{R}\times \mathbb{S}^{n-1})\cap\mathcal{V}(\mathbb{R}\times \mathbb{S}^{n-1})$, and, by applying the Hahn-Banach theorem, one establishes the existence of a spherical representation $g\in\mathcal{D}'(\mathbb{R}\times \mathbb{S}^{n-1})$ for $f$.

We now treat the ultradistribution case. We consider
$${\mathcal V}^{*}(\mathbb R\times \mathbb S^{n-1}):={\mathcal V}(\mathbb R\times \mathbb S^{n-1})\cap {\mathcal E}^{*}(\mathbb R\times \mathbb S^{n-1}),$$
a closed subspace of ${\mathcal E}^{*}(\mathbb R\times \mathbb S^{n-1})$. It is clear that (\ref{sreq2}) maps $\mathcal{E}^{\ast}(\mathbb{R}^{n})$ continuously into ${\mathcal V}^{*}(\mathbb R\times \mathbb{S}^{n-1})$, but whether this mapping is surjectieve or not is not evident. The next theorem gives a partial answer to this question, which allows one to consider spherical representations of ultradistributions. We associate the weight sequence 
$$
N_{p}=\sqrt{p!M_{p}}
$$
to $M_p$. Note that $N_{p}\subset M_{p}$ in the Roumieu case, while $N_{p}\prec M_{p}$ in the Beurling case. The symbol $\dagger$ stands for $\{N_p\}$ if $\ast=\{M_p\}$, while when $\ast=(M_p)$ we set $\dagger=(N_p)$.

\begin{theorem}
\label{srth1}
 Suppose that $M_p$ satisfies $(M.0)$, $(M.1)$, and $(M.2)$.
\begin{itemize}
\item [$(i)$] The linear mapping $\Phi\to \varphi$, where $\varphi$ is given by (\ref{sreq3}), maps continuously ${\mathcal V}^{\dagger}(\mathbb R\times \mathbb S^{n-1})$ into $\mathcal{E}^{\ast}(\mathbb{R}^{n})$. 
\item [$(ii)$] Any ultradistribution $f\in {\mathcal{E}^{\ast}}'(\mathbb{R}^{n})$ admits a spherical representation from ${\mathcal{E}^{\dagger}}'(\mathbb{R}\times\mathbb{S}^{n-1})$; more precisely, one can always find $g\in{\mathcal{E}^{\dagger}}'(\mathbb{R}\times\mathbb{S}^{n-1})$ such that $(\ref{sreq1})$ holds for all $\varphi\in\mathcal{E}^{\dagger}(\mathbb{R}^{n})$.
\end{itemize}
\end{theorem}

If $M_p$ additionally satisfies $(M.3)'$, one obviously obtains an analogous version of Theorem \ref{srth1} for $\mathcal{D}^{\ast}(\mathbb{R}^{n})$ and ${\mathcal{D}^{\ast}}'(\mathbb{R}^{n}).$ When $\ast=\{p!\}$, the sequence $N_p$ becomes equivalent to $p!$. We thus obtain the following corollary for real analytic functions and analytic functionals.

\begin{corollary}
\label{src1} The linear mapping (\ref{sreq2}) is a (topological) isomorphism between the space the real analytic functions $\mathcal{A}(\mathbb{R}^{n})$ and ${\mathcal V}^{\{p!\}}(\mathbb R\times \mathbb S^{n-1})$. Furthermore, any analytic functional $f\in {\mathcal{A}}'(\mathbb{R}^{n})$ has a spherical representation $g\in{\mathcal{A}}'(\mathbb{R}\times\mathbb{S}^{n-1})$, so that $(\ref{sreq1})$ holds for all $\varphi\in\mathcal{A}(\mathbb{R}^{n})$.
\end{corollary}

The rest of this section is devoted to give a proof of Theorem \ref{srth1}. Note that $(ii)$ is a consequence of $(i)$ and the Hahn-Banach theorem (arguing as in the distribution case). In order to show $(i)$ we first need to establish a series of lemmas, some of them are interesting by themselves.

\begin{lemma}
\label{srl1} The space $\mathcal{V}^{\ast}(\mathbb{R}\times \mathbb{S}^{n-1})$ consists of all those $\Phi\in\mathcal{E}^{\ast}(\mathbb{R}\times \mathbb{S}^{n-1})$ whose coefficient functions $a_{k,j}\in\mathcal{E}^{\ast}(\mathbb{R})$ in the spherical harmonic expansion
$$
\Phi(r,\omega)=\sum_{j=0}^{\infty}\sum_{k=1}^{d_j} a_{k,j}(r)Y_{k,j}(\omega)$$
satisfy that $a^{(m)}_{k,j}(0)=0$ for each $m< j$, and $a_{k,j}$ is an even function if $j$ is even and $a_{k,j}$ is an odd function if $j$ is odd. 
\end{lemma}
\begin{proof} Proposition \ref{p sh} ensures that $\Phi$ has the spherical harmonic series expansion. Since $\Phi\in \mathcal{E}^{\ast}_{e}(\mathbb{R}\times \mathbb{S}^{n-1})$ we must necessarily have that  $a_{k,j}$ is even when $j$ is even and $a_{k,j}$ is odd when $j$ is odd. 
Moreover, the other claim readily follows from the fact that for each $m\in\mathbb{N}$
$$\sum_{j=0}^{\infty}\sum_{k=1}^{d_j}a^{(m)}_{k,j}(0)Y_{k,j}(\omega)$$
needs to be the restriction to the sphere of a homogeneous polynomial of degree $m$, as for it $a^{(m)}_{k,j}(0)$ needs to be zero if $j>m$.

\end{proof}

The latter suggests to study for each $j$ ultradifferentiable functions having the same properties as the coefficient functions $a_{k,j}$ from Lemma \ref{srl1}. Define the closed subspace
$$\mathcal{X}^{\ast}_{j}=\{\varphi\in{\mathcal E}^{*}(\mathbb R):\: \varphi^{(m)}(0)=0,\  \forall m<j\}.$$ 

\begin{lemma}\label{srl2}
Let $j\in\mathbb{N}$ and suppose $M_{p}$ satisfies $(M.0)$, $(M.1)$, and $(M.2)'$. The mapping
$$\phi\mapsto\psi,\quad \psi(r):=\frac{\phi(r)}{r^j},$$
 is an isomorphism of TVS from $\mathcal{X}^{\ast}_{j}$ onto ${\mathcal E}^{*}(\mathbb R)$. Moreover, giving a compact $K\subset \mathbb{R}$ and an arbitrary neighborhood $U$ of $K$ with compact closure, there is a constant $\ell$, only depending on $K$, $U$, and $M_{p}$ (but not on $j$), such that
\begin{equation}
\label{sreq5}
\|\psi\|_{\mathcal{E}^{\{M_{p}\},\ell h}(K)}\leq C_{h,U}\|\phi\|_{\mathcal{E}^{\{M_{p}\},h}(\overline{U})}, \quad \forall \phi\in \mathcal{X}^{\ast}_{j}.
\end{equation}
\end{lemma}
\begin{proof} The inverse mapping is obviously continuous, so it suffices to prove the last assertion. In order to treat the non-quasianalytic and quasianalytic cases simultaneously via a Paley-Wiener type argument, we use a H\"ormander analytic cut-off sequence \cite{Hormander,Petzsche2}. So, find a sequence $\chi_{p}\in\mathcal{D}(\mathbb{R})$ such that 
$\chi_{p}\equiv 1$ on $K$, $\chi_{p}(x)=0$ off $U$, and 
$$
\|\chi_{p}^{(m)}\|_{L^{\infty}(\mathbb{R})}\leq C(\ell_{1}p)^{m}, \quad m\leq p. 
$$
By $(M.0)$ and $(M.1)$, we find with the aid of the Leibniz formula a constant $\ell_2$ such that the Fourier transform of $\phi_{p}=\chi_{p}\phi$ satisfies
\begin{equation}\label{sreq6}
|u^{p} \hat{\phi}_p(u)|\leq C' M_{p} (\ell_{2}h)^{p}\|\phi\|_{\mathcal{E}^{\{M_{p}\},h}(\overline{U})}, \quad u\in\mathbb{R}, \ p\in\mathbb{N},
\end{equation}
for all $\phi\in\mathcal{E}(\mathbb{R})$ with $C'=C'_{h,U}$. Consider now $\phi\in\mathcal{X}^{\ast}_{j}$ and the corresponding $\psi$. Setting $\psi_{p}=\chi_p \psi$, and Fourier transforming $r^j\psi_p(r)=\phi_p(r)$, we get $\hat {\psi}^{(j)}_{p}(u)=i^{j}\hat{\phi}_p(u)$. Thus, using the assumption $\phi^{(m)}(0)=0$ for $m<j$, we obtain
\begin{align*}
\hat\psi_{p}(u)&=i^{j}\int_{-\infty}^u \int_{-\infty}^{t_{j-1}}\dots\int_{-\infty}^{t_{1}} \hat\phi_{p}(t_{1})dt_{1}\dots dt_{j}
\\
&
=(-i)^{j}\int_{u}^{\infty}\int^{\infty}_{t_{j-1}}\dots\int^{\infty}_{t_{1}}\hat\phi_{p}(t_1)dt_{1}\dots dt_{j}.
\end{align*}
Employing this expression for $\hat\psi_{p}$ and the fact that $\psi=\psi_p$ on $K$, one readily deduces (\ref{sreq5}) from (\ref{sreq6}) after applying the Fourier inversion formula and $(M.2)'$.
 \end{proof}

Denote as ${\mathcal E}_e^{*}(\mathbb R)$ the subspace of even $\ast$-ultradifferentiable functions.

\begin{lemma}\label{srl3} Assume $M_{p}$ satisfies $(M.0)$, $(M.1)$, and $(M.2)$. The linear mapping
$$
\phi\mapsto \psi, \quad \psi(r)=\phi(\sqrt{|r|}),
$$
maps continuously ${\mathcal E}_e^{\dagger}(\mathbb R)$ into ${\mathcal E}^{*}(\mathbb R)$. 
\end{lemma}
\begin{proof} We only give the proof in the non-quasianalytic case, the quasianalytic case can be treated analogously by using an analytic cut-off sequence exactly as in the proof of Lemma \ref{srl2}. Take an arbitrary even function $\phi\in\mathcal{D}^{\dagger}(K)$ with $\|\phi\|_{\mathcal{E}^{\{\sqrt{p!M_p}\},h}(K)}=1$ and set $\psi(r^2)=\phi(r)$. We have 
\begin{equation}
\label{sreq7} |u^{2p+1} \hat{\phi }(u)|\leq |K| h^{2p+1} \sqrt{(2p+1)!M_{2p+1}}\leq C'_{h} (\ell h^2)^{p}p!M_p.
\end{equation}
with $C'_{h}=h|K|AH\sqrt{M_1}$ and $\ell=(2H)^{3/2}$, because of $(M.2)$. Consider

$$
|u^p \hat\psi(u)|=\left|u^p\int_{-\infty}^{\infty} \phi(\sqrt{|r|})e^{iru}dr\right|=4\left|u^p\int_{0}^{\infty} y\phi(y) \cos(y^2u)dy\right|.
$$
Integrating by parts the very last integral, we arrive at
$$
|u^{p} \hat\psi(u)|=2 \left|u^{p-1}\int_{0}^{\infty}\phi'(y)\sin(y^2u)dy\right|.
$$
Note that $\phi'$ is odd and so $\phi'(0)=0$. Iterating this integration by parts procedure, we find that  
\begin{equation}
\label{sreq8}|u^p \hat\psi(u)|=\frac{1}{2^{p-1}}\left|\int_0^{\infty} \mathcal{L}^{p-1}(\phi') G(y^2u)dy\right|\leq |K|2^{1-p} \|\mathcal{L}^{p-1}(\phi')\|_{L^{\infty}(K)}
\end{equation}
where $G(t)=\sin t$ or $G(t)=\cos t$ and the differential operator $\mathcal{L}$ is given by
$$(\mathcal{L} \varphi)(y) =\frac{d}{dy}\left(\frac{\varphi(y)}{y}\right).$$
Note that $\mathcal{L}$ and their iterates are well-defined for smooth odd functions. Our problem then reduces to estimate $\mathcal{L}^{p-1}(\phi')$. Let $\eta_{p}$ be the Fourier transform of $\mathcal{L}^{p-1}(\phi')$, then 
$$
|\eta_{p}(u)|=|(T^{p-1}\widehat{(\phi')} (u)|,
$$
where 
$$
(T\kappa)(u)=
\begin{cases}
\int_{u}^{\infty} t\kappa(t)dt & \quad \mbox{for }u>0\\
\int^{u}_{-\infty} t\kappa(t)dt & \quad \mbox{for }u<0 \ .
\end{cases}
$$
The inequality (\ref{sreq7}) then gives $(1+|u|^{2})\|\eta_p\|_{L^{\infty}(\mathbb{R})}\leq C''_{h}(\ell h^{2})^{p}M_p $. Fourier inverse transforming and using (\ref{sreq8}), we see that $\|\psi^{(p)}\|_{L^{\infty}(\mathbb{R})}\leq C_h (\ell H h^2)^{p} M_p$, which shows the claimed continuity.
\end{proof}

We need one more lemma. We denote as $B(0,r)$ the Euclidean ball with radius $r$ and center at the origin.

\begin{lemma}
\label{srl4} Given $r<1$ there are constants $L=L_r$ and $C=C_r$ such that for any  homogeneous harmonic polynomial $Q$ on $\mathbb{R}^{n}$ one has
$$\|\partial^{\alpha}Q\|_{L^\infty(B(0,r))}\leq  C L^{|\alpha|} \alpha! \|Q_{|\mathbb{S}^{n-1}}\|_{L^2(\mathbb{S}^{n-1})}.$$
\end{lemma}
\begin{proof} By a result of Komatsu, one has that there is $L$, depending only on $r$, such that
$$
\|\varphi\|_{\mathcal{E}^{\{p!\},Lh}(\overline{B(0,r)})}\leq C_h \sup_{p\in\mathbb{N}}
\frac{\|\Delta^p \varphi\|_{L^2(B(0,1))}}{h^{2p}(2p)!}.
$$
(This actually holds for more general elliptic operators \cite{Komatsu0}.) The estimate then follows by taking $h=1$, $\varphi=Q$, using that $Q$ is harmonic, and writing out the integral in polar coordinates. 
\end{proof}

We are ready to prove Theorem \ref{srth1}:

\begin{proof}[Proof of Theorem \ref{srth1}] We have already seen that $(ii)$ follows from $(i)$.  Let $\Phi\in V^{\dagger}(\mathbb{R}\times \mathbb{S}^{n-1})$ and set $\varphi$ as in (\ref{sreq3}). Since the change of variables $(r,\omega)\mapsto r\omega$ is analytic and invertible away from $r=0$, it is enough to work with ultradifferentiable norms in a neighborhood of $x=0$. Specifically, we estimate the ultradifferentiable norms of $\varphi$ on the ball $B(0,1/2)$. Expand $\Phi$ as in Lemma \ref{srl1} and assume that (cf. Proposition \ref{p sh})
$$
\|a_{k,j}\|_{{\mathcal E}^{\{\sqrt{p!M_p}\},h}([-1,1])}\leq e^{-M\left(\frac{j}{h}\right)}, \quad \forall j,k.
$$
Combining Lemma \ref{srl2} and Lemma \ref{srl3}, we can write
$$\frac{a_{k,j}(r)}{r^j}=b_{k,j}(r^2) \quad \mbox{with }b_{k,j}\in \mathcal{E}^{\ast}(\mathbb{R}) $$
and 
$$
\|b_{k,j}\|_{{\mathcal E}^{\{M_p\},\ell_1h^2}([-1/2,1/2])}\leq C'_he^{-M\left(\frac{j}{h}\right)}, \quad \forall j,k.
$$
where the constant $\ell_{1}$ does not depend on $h$. Therefore,
$$
\varphi(x)=\varphi(r\omega)=\sum_{j=0}^{\infty}\sum_{k=1}^{d_j} B_{k,j}(x)P_{k,j}(x)
$$
where $B_{k,j}(x)=b_{k,j}(|x|^2)$ and $P_{k,j}$ is the harmonic polynomial whose restriction to the unit sphere is $Y_{k,j}$. Since the mapping $x\mapsto |x|^2$ is analytic, the function $B_{k,j}$ is $\ast$-ultradifferentiable and furthermore we can find another constant $\ell_2$ such that

$$
\|B_{k,j}\|_{{\mathcal E}^{\{M_p\},\ell_2 h^2}(B(0,1/2))}\leq C_he^{-M\left(\frac{j}{h}\right)}, \quad \forall j,k.
$$
Suppose $p!\leq C_{h_1} h_1^p M_p$. By $(M.1)$, Lemma \ref{srl4}, and the Leibniz formula,  
\begin{align*}
\|\partial^{\alpha} \varphi\|_{L^{\infty}(B(0,1/2))}\leq  C C_{h_1}C_{h} (L h_1+\ell_2 h^2)^{|\alpha|}M_{|\alpha|} \sum_{j=0}^{\infty}d_je^{-M\left(\frac{j}{h}\right)}
\end{align*}
which completes the proof of Theorem \ref{srth1} because $\log t=o(M(t))$ and $d_{j}=O(j^{n-2})$.
\end{proof}

We end this section with two remarks. Remark \ref{srrk2} poses an open question.
\begin{remark}
\label{srrk1} The technique from this section leads to a new proof of Proposition \ref{srp1} as well.
\end{remark}

\begin{remark}
\label{srrk2} Whether  Theorem \ref{srth1} and Lemma \ref{srl3} hold true or false with $\dagger=\ast$ is an open question. Notice that this holds when $\ast=\{p!\}$ (Corollary \ref{src1}).
\end{remark}

\section{Rotation Invariant Ultradistributions}
\label{section rotation invariant}
We now turn our attention to the characterization of rotation invariant ultradistributions via spherical means. 

We begin with the case of ultradistributions from ${\mathcal{E}^{\ast}}'(\mathbb{R}^{n})$. We say that $f\in{\mathcal{E}^{\ast}}'(\mathbb{R}^{n})$ is rotation invariant if $f(x)=f(T x)$ for all $T\in SO(n)$, the special orthogonal group, namely, if for every rotation $T$ and every $\varphi\in{\mathcal{E}^{\ast}}(\mathbb{R}^{n})$
$$\langle f(x),\varphi(x)\rangle=\langle f(x), \varphi(T^{-1}x)\rangle.$$
Note that the mapping $\varphi\to \varphi_{S}$, where $\varphi_S$ is its spherical mean, is continuous from ${\mathcal{E}^{\ast}}(\mathbb{R}^{n})$ into itself. This can easily be viewed from the alternative expression \cite{HelgasonBook} 
$$
\varphi_{S}(x)=\int_{SO(n)} \varphi(Tx) dT,
$$
where $dT$ stands for the normalized Haar measure of $SO(n)$. The spherical mean of $f\in{\mathcal{E}^{\ast}}'(\mathbb{R}^{n})$ is the ultradistribution $f_{S}\in{\mathcal{E}^{\ast}}'(\mathbb{R}^{n})$ defined by
$$\langle f_S,\varphi\rangle=\langle f, \varphi_S\rangle.$$
Clearly $f_S$ is rotation invariant. All these definitions also apply to $f\in{\mathcal{D}^{\ast}}'(\mathbb{R}^{n})$ if $M_p$ is non-quasianalytic.

\begin{theorem} \label{smth1} Suppose $M_p$ satisfies $(M.0)$, $(M.1)$, and $(M.2)'$. Then,
$f\in {{\mathcal  E}^*}'(\mathbb R^n)$ is rotation invariant if and only if $f=f_S$. 
\end{theorem}

\begin{proof} We only need to show that if $f$ is rotation invariant then $f=f_S$. Furthermore, the general case actually follows from that of analytic functionals. In fact, suppose the theorem is true for $\ast=\{p!\}$. Since $ \mathcal{A}(\mathbb{R}^{n})$ is densely injected into ${{\mathcal  E}^*}(\mathbb R^n)$, we have that $f\in {{\mathcal  E}^*}'(\mathbb R^n)$ is rotation invariant if and only if it is rotation invariant when seen as an analytic functional. Furthermore, taking spherical mean commutes with the embedding 
${{\mathcal  E}^*}'(\mathbb R^n)\to \mathcal{A}'(\mathbb{R}^{n})$, whence our claim follows.

Suppose that $f\in\mathcal{A}'(\mathbb{R})$ is rotation invariant. Applying Corollary \ref{src1} we can find a spherical representation $g\in\mathcal{A}'(\mathbb{R}\times \mathbb{S}^{n-1})$ for $f$. Using Proposition \ref{p sh} we can expand $g$ as 
\begin{equation}
\label{smeq1}
g(r,\omega)=\sum_{j=0}^{\infty}\sum_{k=1}^{d_j} c_{k,j}(r)\otimes Y_{k,j}(\omega)
\end{equation}
with convergence in $\mathcal{A}'(\mathbb{R}\times \mathbb{S}^{n-1})$ where $c_{k,j}$ are one-dimensional analytic functionals. Notice that if we also expand the polar coordinate expression of $\varphi\in \mathcal{E}^{\ast}(\mathbb{R}^{n})$ as
$\varphi(r\omega)=\sum_{j=0}^{\infty}\sum_{k=1}^{d_j} a_{k,j}(r) Y_{k,j}(\omega)$,
we obtain that $\varphi_{S}(r\omega)= |\mathbb{S}^{n-1}|^{-1/2}a_{0,0}(r)=a_{0,0}(r)Y_{0,0}(\omega)$. The latter holds because $\int_{\mathbb{S}^{n-1}}Y_{k,j}(\omega)d\omega=0$ for $j\geq 1$, which follows from the mean value theorem for harmonic functions. Thus, $c_{0,0}\otimes Y_{0,0}$ is a spherical representation for $f_{S}$. The result would then follow if we show that $c_{0,0}\otimes Y_{0,0}$ is also a spherical representation of $f$. By Lemmas \ref{srl1}-\ref{srl3} and the expansion (\ref{smeq1}), this would certainly be the case if we show that 
\begin{equation}
\label{smeq2}
\langle f(x), |x|^{2m}Q(x) \rangle=0
\end{equation}
for every $m\in\mathbb{N}$ and every harmonic homogeneous polynomial $Q$ of degree $j\geq1$.
Since every such $Q$ can be written \cite[Prop.~5.31]{Axler} as 
$$
Q(x)= \int_{\mathbb{S}^{n-1}} Q(\omega) Z_{j}(x,\omega) d\omega,
$$
where $Z_{j}(x,\omega)$ is the zonal spherical harmonic of degree $j$, we have that
$$
\langle f(x), |x|^{2m}Q(x) \rangle=\int_{\mathbb{S}^{n-1}}Q(\omega)P_{j}(\omega)d\omega
$$
with 
$$
P_{j}(\omega):= \langle f(x), |x|^{2m}Z_{j}(x,\omega) \rangle, \quad \omega\in\mathbb{S}^{n-1}.
$$
So (\ref{smeq2}) would hold if we show that $P_{j}$
identically vanishes on $\mathbb{S}^{n-1}$ if $j\geq 1$. Observe that $P_j$ is a spherical harmonic of degree $j\geq1$. On the other hand, $Z_{j}(T^{-1}x,\omega)=Z_{j}(x,T\omega)$ for every rotation $T$ \cite[Prop.~5.27]{Axler}, and using the fact that $f$ is rotation invariant, we obtain $P_{j}(T\omega)=P_{j}(\omega)$ for all $\omega\in \mathbb{S}^{n-1}$ and $T\in SO(n)$. Due to the fact that the group $SO(n)$ acts transitively on $\mathbb{S}^{n-1}$, $P_j$ must be a constant function, and hence a spherical harmonic of degree 0. Since the spaces of spherical harmonics of different degrees are mutually orthogonal in $L^{2}(\mathbb{S}^{n-1})$, one concludes that $P_{j}\equiv 0$ if $j\neq0$. This concludes the proof of the theorem.

\end{proof}
 
In the non-quasianalytic case, we can use Theorem \ref{smth1} to recover the result \cite[Thm.~4.4]{Chung} by Chung and Na  quoted at the Introduction. 

\begin{theorem} \label{smc2} Suppose $M_p$ satisfies $(M.1)$, $(M.2)'$, and $(M.3)'$. An ultradistribution $f\in{\mathcal{D}^{\ast}}'(\mathbb{R}^{n})$ is rotation invariant if and only if $f=f_{S}$.
\end{theorem}
\begin{proof}
Using a partition of the unity, we can write any rotation invariant $f$ as a locally finite sum $\sum_{k=1}^{\infty} f_{k}$ with each $f_{k}\in{\mathcal{E}^{\ast}}'(\mathbb{R}^{n})$ being also rotation invariant. By Theorem \ref{smth1} we have $f_{k}=(f_{k})_{S}$, and, consequently, $f_{S}=\sum_{k=1}^{\infty} (f_{k})_{S}= \sum_{k=1}^{\infty} f_{k}=f$. 
\end{proof} 

We now discuss how one can extend Theorem \ref{smth1} in the quasianalytic case (including the hyperfunction case). From now on we assume that $M_p$ satisfies $(M.0)$, $(M.1)$, $(M.2)'$, and $(QA)$. Our next considerations are in terms of sheaves of quasianalytic ultradistributions, we briefly discuss their properties following the approach from \cite{Debrouwere-VindasSQU,Hormander85} (cf. \cite{Schapira} for hyperfunctions).

Let $f \in {\mathcal{E}^\ast}'(\mathbb{R}^n)$ (referred to as a $\ast$-quasianalytic functional hereafter). A compact  $K \subseteq \mathbb{R}^n$ is called a $\ast$-\emph{carrier} of $f$ if $f \in {\mathcal{E}^\ast}'(\Omega)$ for every open neighborhood $\Omega$ of $K$. If $f \in \mathcal{A}'(\mathbb{R}^n)$, it is well-known \cite[Sect.~9.1]{Hormander} that there is a smallest compact $K \subseteq \mathbb{R}^n$ among all the $\{p!\}$-carriers of $f$, the $\{p!\}$-support of $f$ denoted by $\operatorname{supp}_{\mathcal{A}'}f$. It was noticed by H\"ormander  that a similar result basically holds for quasianalytic functionals \cite[Cor. 3.5]{Hormander85}, that is, for any $\ast$-quasianalytic functional there is a smallest $\ast$-carrier, say $\operatorname*{supp}_{{\mathcal{E}^{\ast}}'}f$,  and one has $\operatorname{supp}_{\mathcal{A}'}f=\operatorname*{supp}_{{\mathcal{E}^{\ast}}'}f$. H\"ormander only treats the Roumieu case in \cite{Hormander85}, but his proof can be modified to show the corresponding statement for the Beurling case \cite{Debrouwere-VindasSQU,Heinrich}.

Denote as ${\mathcal{E}^{\ast}}'[K]$ the space of $\ast$-quasianalytic functionals with support in $K$. One can show that there is an (up to isomorphism) unique flabby sheaf $\mathfrak{B}^{\ast}$ whose space of global sections with support in $K$ is precisely ${\mathcal{E}^{\ast}}'[K]$, for any compact of $\mathbb{R}^{n}$. We call $\mathfrak{B}^{\ast}$ the sheaf of $\ast$-quasianalytic ultradistributions. 
When $\ast=\{p!\}$, we simply write $\mathfrak{B}=\mathfrak{B}^{\ast}$, the sheaf of hyperfunctions. Actually, in the Roumieu case the existence of $\mathfrak{B}^{\ast}$ can be established exactly as for hyperfunctions with the aid of H\"{o}rmander support theorem by using the Martineau-Schapira method \cite{Schapira}. Details for the Beurling case, which require a subtler treatment, will appear in the forthcoming paper \cite{Debrouwere-VindasSQU}. Since it is important for us, we mention that on any bounded open set $\Omega$ the sections of $\mathfrak{B}^{\ast}$ are given by the quotient spaces
\begin{equation}
\label{smeq3} \mathfrak{B}^{\ast}(\Omega)={\mathcal{E}^{\ast}}'[\overline{\Omega}]/{\mathcal{E}^{\ast}}'[\partial\Omega],
\end{equation}
 which reduces to the well-known Martineau theorem in the case of hyperfunctions. Finally, we call the space of global sections $\mathfrak{B}^{\ast}(\mathbb{R}^{n})$ the space of $\ast$-quasianalytic ultradistributions on $\mathbb{R}^{n}$ (hyperfunctions if $\ast=\{p!\}$).
 
The operation of taking spherical mean preserves the space ${\mathcal{E}^{\ast}}'[K]$ if $K$ is a rotation invariant compact set. Because of (\ref{smeq3}), we can define the spherical mean $f_{S}\in\mathfrak{B}^{\ast}(\Omega)$ of $f\in\mathfrak{B}^{\ast}(\Omega)$ in a canonical manner if $\Omega$ is a bounded rotation invariant open subset of $\mathbb{R}^{n}$, namely, if $f=[g]$ with $g={\mathcal{E}^{\ast}}'[\overline{\Omega}]$, we define $f_{S}=[g_{S}]$. Using the sheaf property, one extends the definition $f_{S}\in\mathfrak{B}^{\ast}(\mathbb{R}^{n})$ for all $f\in\mathfrak{B}^{\ast}(\mathbb{R}^{n})$. We say that $f\in\mathfrak{B}^{\ast}(\mathbb{R}^{n})$ is rotation invariant if its restriction to $\Omega$ is rotation invariant for any rotation invariant bounded open set $\Omega$ (the latter makes sense because of (\ref{smeq3})). Theorem \ref{smth1} implies the following generalization:

\begin{theorem} \label{smc3} Suppose $M_p$ satisfies $(M.0)$, $(M.1)$, $(M.2)'$, and $(QA)$. A quasianalytic ultradistribution $f\in\mathfrak{B}^{\ast}(\mathbb{R}^{n})$ is rotation invariant if and only if $f=f_{S}$.
\end{theorem}

We point out that Theorem \ref{smc3} extends \cite[Thm.~5.7]{Chung}, which was obtained for hyperfunctions.

\end{document}